\documentclass[11pt]{amsart}
\usepackage{amsmath, amssymb}
\usepackage[T1]{fontenc}   
\usepackage{stmaryrd}
\usepackage{youngtab}
\usepackage[english]{babel}

\newcommand{\tdun}[1]{\begin{picture}(10,5)(-2,-1)
\put(0,0){\circle*{2}}
\put(3,-2){\tiny #1}
\end{picture}}

\newcommand{\tddeux}[2]{\begin{picture}(12,5)(0,-1)
\put(3,0){\circle*{2}}
\put(3,5){\circle*{2}}
\put(3,0){\line(0,1){5}}
\put(6,-2){\tiny #1}
\put(6,3){\tiny #2}
\end{picture}}

\newcommand{\tdtroisun}[3]{\begin{picture}(20,12)(-5,-1)
\put(3,0){\circle*{2}}
\put(6,7){\circle*{2}}
\put(0,7){\circle*{2}}
\put(-0.65,0){$\vee$}
\put(5,-2){\tiny #1}
\put(9,5){\tiny #2}
\put(-5,5){\tiny #3}
\end{picture}}

\newcommand{\tdtroisdeux}[3]{\begin{picture}(12,15)(-2,-1)
\put(0,0){\circle*{2}}
\put(0,5){\circle*{2}}
\put(0,10){\circle*{2}}
\put(0,0){\line(0,1){5}}
\put(0,5){\line(0,1){5}}
\put(3,-2){\tiny #1}
\put(3,3){\tiny #2}
\put(3,9){\tiny #3}
\end{picture}}

\input{xy}
\xyoption{all}

\newtheorem{defi}{\indent Definition}
\newtheorem{lemma}[defi]{\indent Lemma}
\newtheorem{cor}[defi]{\indent Corollary}
\newtheorem{theo}[defi]{\indent Theorem}
\newtheorem{prop}[defi]{\indent Proposition}

\newcommand{\N}{\mathbb{N}}

\newcommand{\dpos}{\mathbf{dp}}
\newcommand{\dqp}{\mathbf{dqp}}

\newcommand{\sqp}{\mathbf{sqp}}
\newcommand{\tqp}{\mathbf{tqp}}

\newcommand{\h}{\mathcal{H}}
\newcommand{\E}{{\mathcal E}}
\newcommand{\I}{{\mathcal I}}

\begin{document}

\author{Lo\"\i c Foissy}
\address{LMPA Joseph Liouville\\
Universit\'e du Littoral C\^ote d'opale\\ Centre Universitaire de la Mi-Voix\\ 50, rue Ferdinand Buisson, CS 80699\\ 62228 Calais Cedex, France\\
email {foissy@lmpa.univ-littoral.fr}}

\author{Claudia Malvenuto}
\address{Dipartimento di Matematica\\ Sapienza Universit\`a
di Roma\\ P.le A. Moro 5\\ 00185, Roma, Italy\\
email {claudia@mat.uniroma1.it}}

\author{Fr\'ed\'eric Patras}
\address{UMR 7351 CNRS\\
        		Universit\'e de Nice\\
        		Parc Valrose\\
        		06108 Nice Cedex 02
        		France\\
        		email {patras@unice.fr}}
\title{A theory of pictures for quasi-posets}
\date{}

\begin{abstract}
The theory of pictures between posets is known to encode much of the combinatorics of symmetric group representations and related topics such as Young diagrams and tableaux. Many reasons, combinatorial (e.g. since semi-standard tableaux can be viewed as double quasi-posets) and topological (quasi-posets identify with finite topologies) lead to extend the theory to quasi-posets. This is the object of the present article.
\end{abstract}

\maketitle

\section*{Introduction}\label{Intro}

The theory of pictures between posets is known to encode much of the combinatorics of symmetric group representations and related topics such as preorder diagrams and tableaux.
The theory captures for example the Robinson-Schensted (RS) correspondence or the Littlewood-Richardson formula, as already shown by Zelevinsky in the seminal article \cite{zelevinsky1981generalization}.
Recently, the theory was extended to double posets (pairs of orders coexisting on a given finite set -- hereafter, ``order'' means ``partial order''; an order on $X$ defines a poset structure on $X$) and developed from the point of view of combinatorial Hopf algebras which led to new advances in the field \cite{MR2,foissy2013algebraic,foissy2013plane,foissy2014deformation}.

In applications, a fundamental property that has not been featured enough, is that often pictures carry themselves implicitly a double poset structure. A typical example is given by standard Young tableaux, which can be put in bijection with certain pictures (this is one of the nicest way in which their appearance in the RS correspondence can be explained \cite{zelevinsky1981generalization}) and carry simultaneously a poset structure (induced by their embeddings into $\N\times\N$ equipped with the coordinate-wise partial order) and a total order (the one induced by the integer labelling of the entries of the tableaux).

However, objects such as tableaux with repeated entries, such as semi-standard tableaux, although essential, do not fit into this framework. They should actually be thought of instead as double quasi-posets (pairs of preorders on a given finite set): the first preorder is the same than for standard tableaux (it is an order), but the labelling by (possibly repeated) integers is naturally captured by a preorder on the entries of the tableau (the one for which two entries are equivalent if they have the same label and else are ordered according to their labels).

Besides the fact that these ideas lead naturally to new results and structures on preorders, other observations and motivations have led us to develop on systematic bases in the present article a theory of pictures for quasi-posets. Let us point out in particular recent developments (motivated by applications to multiple zeta values, Rota-Baxter algebras, stochastic integrals... \cite{KEFsmf,ebrahimi2015flows,ebrahimi2015exponential}) 
that extend to surjections \cite{Novelli,novelli2,foissy2,foissy3}
the theory of combinatorial Hopf algebra structures on permutations \cite{MR,foissy2007bidendriform}. New results on surjections will be obtained in the last section of the article.

Lastly, let us mention our previous works on finite topologies (equivalent to quasi-posets) \cite{FM,FMP} (see also \cite{fauvet2015hopf,fauvet2016operads} for recent developments) which featured the two products defined on finite topologies by disjoint union and the topological join product. 
The same two products, used simultaneously, happen to be the ones that define on double quasi-posets an algebra structure (and actually self-dual Hopf algebra) structure extending the usual one on double posets.

The article is organized as follows.
Section \ref{dqp} introduces double quasi-posets.
Sections \ref{adqp} and \ref{Hopf} introduce and study Hopf algebra structures on double quasi-posets.
Section \ref{pprs} defines pictures between double quasi-posets. Due to the existence of equivalent elements for both preorders
of a double quasi-poset, the very notion of pictures is much more flexible than for double posets. From Section \ref{psdu} onwards,
we focus on the algebraic structures underlying the theory of pictures for double quasi-posets. Section \ref{psdu} investigates duality phenomena and shows that pictures define a symmetric Hopf pairing on the Hopf algebra of double quasi-posets. Section \ref{ip} addresses the question of internal products, generalizing the corresponding results on double posets. Internal products (by which we mean the existence of an associative product of double posets within a given cardinality) are a classical property of combinatorial Hopf algebras. Once again, the rich structure of double quasi-posets allows for some flexibility in the definitions, and we introduce two internal associative products extending the one on double posets and permutations.
Section \ref{ps} investigates the restriction of the internal products to surjections. A product  different from the usual composition of surjections and of the one on the Solomon-Tits algebra emerges naturally from the theory of pictures.

\vspace{0.5cm}
\textbf{Notations.} 
Recall that a packed word  is a word over the integers (or any isomorphic strictly ordered set) containing the letter 1 and such that, if the letter $i>1$ appears, then all the letters between $1$ and $i$ appear (e.g. $21313$ is packed but not $2358223$). We write $\E_n$ for the set of packed words of length $n$; the subset $\E_n(k)$ of packed words of length $n$ with $k$ distinct letters identifies with the set of surjections from $[n]$ to $[k]$ when the latter are represented as a packed word (by writing down the sequence of their values on $1,\dots,n$).
Let us write $\I_n$ for increasing packed words (such as $11123333455$) (resp. $\I_n(k)$ for packed words with $k$ different letters). Increasing packed words of length $n$ are in bijection with compositions ${\mathbf n}=(n_1,\dots,n_k)$, $n_1+\dots+n_k=n$, of $n$, by counting the number of $1s$, $2s$... (The previous increasing packed word is associated to the composition $(3,1,4,1,2)$). 

All the algebraic structures (algebras, vector spaces...) are defined over a fixed arbitrary ground field $k$.

\section{Double quasi-posets}\label{dqp}
In the article, order means partial order. We say equivalently that an order is strict or total.
Preorders are defined by relaxing the antisymmetry condition, making possible $x\leq y$ and $y\leq x$ for $x\not=y$. A set equipped with a preorder is called a quasi-poset.
Finite quasi-posets identify with finite topologies, a classical result due to Alexandroff \cite{Alexandroff} revisited from the point of view of combinatorial Hopf algebras in \cite{FM,FMP}.

\textbf{Notations.} Let $\leq_1$ be a preorder on a set $A$. We define an equivalence relation on $A$ by:
$$\forall i,j\in A,\: i\sim_1 i\mbox{ if } i\leq_1 j\mbox{ and }j\leq_1 i.$$
We shall write $i<_1 j$ if $i\leq_1 j$ and not $i\sim_1 j$.

\begin{defi}
A double quasi-poset is a triple $P=(V(P),\leq_1,\leq_2)$ where $V(P)$ is a finite set, and $\leq_1$, $\leq_2$ are two preorders on $V(P)$.
The set of (isoclasses of) double quasi-posets is denoted by $\dqp$. The vector space generated by $\dqp$ is denoted by $\h_\dqp$.
\end{defi}
In practice, one can always assume that $V(P)=[n]:=\{1,\dots ,n\}$. We denote by $\dqp(n)$ double quasi-posets with $n$ elements (the same notation will be used for other families of objects without further comments).

\begin{defi}
Let $P,Q\in\dqp$. A morphism between $P$ and $Q$ is a doubly increasing bijection, i.e. a bijection $f$ between $V(P)$ and $V(Q)$ such that
$$i\leq_1 j\Rightarrow f(i)\leq_1f(j),$$
$$i\leq_2 j\Rightarrow f(i)\leq_2f(j).$$
\end{defi}

The morphism $f$ is an isomorphism (resp. an automorphism when $P=Q$) if and only if
$$i\leq_1 j\Leftrightarrow f(i)\leq_1f(j),$$
$$i\leq_2 j\Leftrightarrow f(i)\leq_2f(j).$$
We write $Aut(P)$ for the group of automorphisms of $P$.

\begin{defi}
A double quasi-poset $P$ is special (resp. strict special) if $\leq_2$ is a total preorder, that is to say:
$$\forall i,j\in V(P),\: i\leq_2 j\mbox{ or }j\leq_2 i,$$
(resp. a total order).
The set of (isoclasses) of special double quasi-posets is denoted by $\sqp$. The vector space generated by $\sqp$ is denoted by $\h_\sqp$.
\end{defi}

Notice that a total preorder on $[n]$ identifies canonically with a surjection, and conversely. This is best explained through an example indicating the general rule: consider the surjection $f$ from $[5]$ to $[3]$ defined by
$$f(2)=f(4):=1,\ \ f(1):=2,\ \ f(3)=f(5):=3,$$
the corresponding total preorder $\leq_f$ (with a self-explaining notation) is
$$2\sim_f4\leq_f 1\leq_f 3\sim_f 5.$$
We will represent both $f$ and $\leq_f$ by the packed word associated to the sequence of values of $f$ on $1,\dots,5$: $f=21313$. 

\begin{defi}
A double quasi-poset $P$ is trivial if $\leq_1$ is the trivial preorder (i.e. two distinct elements are never comparable for $\leq_1$).
The set of trivial double quasi-posets is denoted by $\tqp$. It is in bijection with the set of (isoclasses of) quasi-posets. The vector space generated by $\tqp$ is denoted by $\h_\tqp$.
\end{defi}

\begin{defi} Let $P\in \dqp$. If both $\leq_1$ and $\leq_2$ are orders (resp. if $\leq_2$ is strict), we shall say that $P$ is a double poset (resp. special double poset).
The set of (isoclasses of) double posets is denoted by $\dpos$ and the space generated by $\dpos$ is denoted by $\h_\dpos$.
\end{defi}

We graphically represent any special double quasi-poset $P$ by the reduced Hasse graph of $\leq_1$ (reduced means that two equivalent vertices are identified); the second, total, preorder
is given by integer indices on the vertices of this graph. For example, here are special double quasi-posets of cardinality $\leq 2$:
\begin{align*}
&1; \tdun{$1$};\tddeux{$1$}{$2$},\tddeux{$2$}{$1$},\tddeux{$1$}{$1$},\tdun{$1$}\tdun{$2$},\tdun{$1$}\tdun{$1$},
\tdun{$1,2$}\hspace{3mm}, \tdun{$1,1$}\hspace{2mm}.
\end{align*}

Here, $1$ denotes the empty graph, $\tdun{$1,2$}$ \hspace{2mm} (resp. \tdun{$1,1$}\hspace{2mm}) represents a two-elements set $\{a,b\}$ with $a\sim_1 b$ and $a<_2b$ (resp. $a\sim_1 b$ and $a\sim_2b$).
For the first cardinalities, we have:

$$\begin{array}{|c|c|c|c|c|}
\hline n&1&2&3&4\\
\hline\sharp\dqp(n)&1&10&166&5965\\
\hline\sharp\sqp(n)&1&7&74&1290\\
\hline\end{array}$$

\section{Algebra structures on double quasi-posets}\label{adqp}
Let $P,Q\in \dqp$. We define two preorders on $V(P)\sqcup V(Q)$:
\begin{align*}
\forall i,j\in V(P)\sqcup V(Q),\: i\leq_1 j\mbox{ if }&(i,j\in V(P)\mbox{ and } i\leq_1 j) \\
&\mbox{ or }(i,j\in V(Q)\mbox{ and } i\leq_1 j);\\
\: i\leq_2 j\mbox{ if }&(i,j\in V(P)\mbox{ and } i\leq_2 j)\\
&\mbox{ or }(i,j\in V(Q)\mbox{ and } i\leq_2 j)\\
&\mbox{ or }(i\in V(P)\mbox{ and } j\in V(Q)).
\end{align*}
This defines a double quasi-poset denoted by $PQ$. Extending this product by bilinearity, we make $\h_\dqp$ an associative algebra, whose
unit is the empty double quasi-poset $1$. 

\begin{lemma}If $P$ and $Q$ are special, then $PQ$ is special: $\h_\sqp$ is subalgebra of $\h_\dqp$.
If $P$ and $Q$ are trivial, then $PQ$ is trivial: $\h_\tqp$ is subalgebra of $\h_\tqp$.
\end{lemma}
From a topological point of view, the first operation (on $\leq_1$) corresponds to the disjoint union of finite topologies; the second, to the join product \cite{FM,FMP}.
It is often useful to transform finite topologies by removing degeneracies (points that can not be separated). The following definition provides a way of doing so in the context of double preorders. 

\begin{defi}
Let $P$ be a double quasi-poset. 
We call splitting of $P$ and denote by $pos(P)=(V(P),\preceq_1,\preceq_2)$ 
the double poset defined by:
\begin{align*}
\forall i,j\in V(P),\:& i\preceq_1 j\mbox{ if } i<_1 j\mbox{ or } i=j,& i\preceq_2 j\mbox{ if } i<_2 j\mbox{ or } i=j.
\end{align*}
\end{defi}

For example, the splitting of $\tddeux{$1$}{$2,3$}\hspace{3mm}$ is $\tdtroisun{$1$}{$3$}{$2$}$. It follows from the definitions that

\begin{lemma}\label{homalg}
 The splitting map is an algebra map from $\h_\dqp$ to its subalgebra $\h_\dpos$.
\end{lemma}

\section{Hopf algebra structures}\label{Hopf}

\begin{defi}
Let $P$ be a double quasi-poset and let $X\subseteq V(P)$.
\begin{itemize}
\item $X$ is also a double quasi-poset by restriction of $\leq_1$ and $\leq_2$:
we denote this double quasi-poset by $P_{\mid X}$. 
\item We shall say that $X$ is an open set of $P$ if:
$$\forall i,j\in V(P),\: i\leq_1 j\mbox{ and }i\in X\Longrightarrow j\in X.$$
The set of open sets of $P$ is denoted by $Top(P)$.
\item We shall say that $X$ is a preopen set of $P$ if:
$$\forall i,j\in V(P),\: i<_1 j\mbox{ and }i\in X\Longrightarrow j\in X.$$
The set of preopen sets of $P$ is denoted by $Top_<(P)$.
\end{itemize} \end{defi}

\textbf{Remark.} The splitting map does not preserve homotopy types but is well-fitted to the notion of preopen sets: $$Top_<(P)=Top(pos(P)).$$ 

We define two coproducts on $\h_\dqp$ in the following way:
\begin{align*}
\forall P \in \dqp,\: &\Delta(P)=\sum_{O\in Top(P)}P_{\mid V(P)\setminus O}\otimes P_{\mid O},\\
&
\Delta_<(P)=\sum_{O\in Top_<(P)}P_{\mid V(P)\setminus O}\otimes P_{\mid O}.
\end{align*}
\begin{theo}
Both $(\h_\dqp,m,\Delta)$ and $(\h_\dqp,m,\Delta_<)$ are graded, connected Hopf algebras; moreover, $\h_\sqp$, $\h_\dpos$ and $\h_\tqp$
are Hopf subalgebra for both coproducts.
Finally, the splitting map $pos$ is a Hopf algebra morphism and a projection from $(\h_\dqp,m,\Delta_<)$ to $(\h_\dpos,m,\Delta)$.
\end{theo}

The coassociativity of $\Delta$ was proven in \cite{FM}, a similar proof holds for $\Delta_<$. The fact that the two coproducts are algebra maps and the other statements of the Theorem follow from the Lemma \ref{homalg} and from the definitions by direct inspection.\\

The Hopf algebra $\h_\tqp$ identifies with  (one of) the Hopf algebras defined in \cite{FMP} on isoclasses of finite topological spaces and of quasi-posets.\\

\textbf{Remark.} If $\leq_1$ is an order, then $\Delta(P)=\Delta_<(P)$. In particular, $(\h_\dpos,m,\Delta)=(\h_\dpos,m,\Delta_<)$.

\ \par

Let us introduce now the notion of blow up.
Let $P=(V(P),\leq_1,\leq_2)\in\dqp$ and write, for $i\in V(P)$, $P_i:=\{j\in V(P),\ i\sim_1 j\}$. If $P_i\not=\{i\}$ let $\leq^i$ be an arbitrary \it total \rm preorder on $P_i$. We can define a new double quasi-poset $P'=(V(P),\leq_1',\leq_2)$ (the blow up of $P$ along $\leq^i$) by:
$$\forall j\notin P_i,\forall k\in V(P), (j\leq_1' k\Leftrightarrow j\leq_1 k) \ \text{and} \ (j\geq_1' k\Leftrightarrow j\geq_1 k)$$
$$\forall (j,k)\in P_i^2, j\leq_1'k \Leftrightarrow j\leq^i k.$$

\begin{defi}
Any double quasi-poset $P'$ obtained by this process is called an elementary blow up of $P$. A double quasi-poset $Q$ obtained from $P$ by a sequence of elementary blow ups is called a blow up of $P$. We write $B(P)$ for the set of blow ups of $P$.
\end{defi}

For example, the blow ups of $\tddeux{$1$}{$2,3$}\hspace{3mm}$ are $\tddeux{$1$}{$2,3$}\hspace{3mm}$,
$\tdtroisdeux{$1$}{$2$}{$3$}$ and $\tdtroisdeux{$1$}{$3$}{$2$}$.\\

Warning: by definition, blow ups of $P$ have the same element sets than $P$ and their two preorders are defined on $V(P)$. Two isomorphic blow ups of $P$ are equal in $\dqp$, but to keep track of multiplicities, we \it do not \rm identify them inside $B(P)$.

\begin{defi}
Let $P,Q\in\dqp$, we shall say that $P\leq Q$ if $Q$ is isomorphic to $P$ or to a blow up of $P$.
\end{defi}

Equivalently: $P\leq Q$ if there exists
a bijection $f:V(P)\longrightarrow V(Q)$ with the following properties:
\begin{itemize}
\item For all $i,j\in V(P)$, $i$ and $j$ are comparable for $\leq_1$ in $P$ if, and only if, $f(i)$ and $f(j)$ are comparable for $\leq_1$ in $Q$.
\item For all $i,j\in V(P)$, if $i<_1 j$ in $P$, then $f(i)<_1 f(j)$ in $Q$.
\item For all $i,j\in V(P)$, if $f(i)\sim_1 f(j)$ in $Q$, then $i\sim_1 j$ in $P$.
\item For all $i,j \in V(P)$, $i\leq_2 j$ in $P$ if, and only if, $f(i)\leq_2 f(j)$ in $Q$.
\end{itemize}

\begin{lemma}
$\leq$ is an order on $\dqp$.
\end{lemma}

\begin{proof} Indeed, the blow up of a blow up of $P$ is a blow up of $P$. Moreover, a non trivial elementary blow up increases strictly the number of equivalence classes for the relation $\leq_1$. It follows that $P\leq Q$ and $Q\leq P$ imply $P=Q$ in $\dqp$.\end{proof}

\textbf{Example}. Here is the subposet of double quasi-posets greater than $\tdun{$1,2,3$}\hspace{5mm}$:
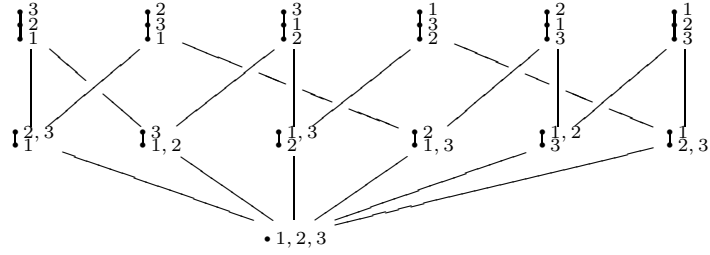
$$\xymatrix{\tdtroisdeux{$1$}{$2$}{$3$}&\tdtroisdeux{$1$}{$3$}{$2$}&
\tdtroisdeux{$2$}{$1$}{$3$}&\tdtroisdeux{$2$}{$3$}{$1$}
&\tdtroisdeux{$3$}{$1$}{$2$}&\tdtroisdeux{$3$}{$2$}{$1$}\\
\tddeux{$1$}{$2,3$}\hspace{2mm}\ar@{-}[u]\ar@{-}[ru]&\tddeux{$1,2$}{$3$}\hspace{2mm}\ar@{-}|(.5)\hole[lu]\ar@{-}[ru]&
\tddeux{$2$}{$1,3$}\hspace{2mm}\ar@{-}[u]\ar@{-}[ru]&\tddeux{$1,3$}{$2$}\hspace{2mm}\ar@{-}|(.33)\hole|(.5)\hole|(.66)\hole[llu]\ar@{-}[ru]&
\tddeux{$3$}{$1,2$}\hspace{2mm}\ar@{-}[u]\ar@{-}[ru]&\tddeux{$2,3$}{$1$}\hspace{2mm}\ar@{-}|(.33)\hole|(.5)\hole|(.66)\hole[llu]\ar@{-}[u]\\
&&\tdun{$1,2,3$}\hspace{5mm} \ar@{-}[llu]\ar@{-}[lu]\ar@{-}[u]\ar@{-}[ru]\ar@{-}[rru]\ar@{-}[rrru]}$$

\textbf{Remark.} Let us take $P,Q\in \dqp$, such that $P\leq Q$. Then:
\begin{align*}
\mbox{ $P$ or $Q$ is special}&\Longleftrightarrow\mbox{$P$ and $Q$ are special}.
\end{align*}

\begin{lemma}
Let us set $b(P):=\sum\limits_{P'\in B(P)}P'$. Then:
$$\Delta\circ b(P)=(b\otimes b)\circ \Delta_<(P).$$
\end{lemma}

\begin{proof} Indeed, $\Delta\circ b(P)$ is a sum of terms $P'_{|O^c}\otimes P'_{|O}$ over open subsets in $Top(P')$. However, by definition of blow ups, open subsets $O$ of $P'$ are preopen sets of $P$ and there is a canonical embedding of the set of pairs $(P'_{|O^c}, P'_{|O})$ in the expansion of $\Delta\circ b(P)$ into the set of pairs 
$$\coprod_{O\in Top_<(P)}B(P_{|O^c})\times B(P_{|O}).$$

Conversely, any element in this last set defines uniquely a pair $(U,P')$ where $U$ is an open set of a blow up $P'$ of $P$. Indeed, let $O\in Top_<(P)$, $T$ be a blow up of $P_{|O}$ and $W$ a blow up of $P_{|O^c}$. Set $U:=T$ and define
the preorder $\leq_1'$ on $P'$ by
$$\forall (i,j)\in (O^c\times O^c) \cup (O\times O),\ i\leq_1' j \Leftrightarrow i\leq_1 j \ \text{in}\ T\ \text{ or}\  W,$$
$$\forall (i,j)\in O^c\times O, \ i\leq_1' j \Leftrightarrow i\leq_1 j \ \text{in}\ P\ \text{when}\ i\not\sim_1 j\ \text{in} \ P, \ i<_1 j\ \text{else}.$$
\end{proof}

\begin{prop}
We consider the map:
\begin{align*}
\Upsilon:&\left\{\begin{array}{rcl}
\h_\dqp&\longrightarrow&\h_\dqp\\
P&\longrightarrow&\displaystyle b(P).
\end{array}\right.
\end{align*}
Then $\Upsilon$ is a Hopf algebra isomorphism from $(\h_\dqp,m,\Delta_<)$ to $(\h_\dqp,m,\Delta)$. 
\end{prop}
\begin{proof}
The blow ups of a product $PQ$ are in a straightforward bijection with the products of blow ups of $P$ and $Q$, the multiplicativity of $\Upsilon$ follows.
That $\Upsilon$ is an isomorphism follows then from its invertibility as a linear map (recall that $b(P)$ is the sum of $P$ with higher order terms for the order $\leq$ on $\dqp$) and from the previous Lemma. \end{proof}

\section{Pictures and patterns}\label{pprs}

Due to the possible existence of equivalent elements for $\leq_1$ or $\leq_2$, the theory of pictures for double quasi-posets (to be introduced in the present section) allows for much more flexibility than the one of pictures for double posets. In particular it allows for various approaches to encode pictorially combinatorial objects such as surjections, tableaux with repeated entries, and so on.
It also provides a new framework (through the notion of patterns, also to be introduced) 
to deal with quotients under Young (and more generally parabolic) subgroups actions. Although the present article is mainly focused on combinatorial Hopf algebra structures, we expect these ideas and the associated algebraic structures to lead to new approaches to these classical topics.

\begin{defi}
Let $P,Q\in \dqp$.
\begin{itemize}
\item A prepicture between $P$ and $Q$ is a bijection $f:V(P)\longrightarrow V(Q)$ such that:
\begin{align*}
\forall i,j \in V(P),\:&i<_1 j\Longrightarrow f(i)<_2 f(j),& &f(i)<_1f(j)\Longrightarrow i<_2 j.
\end{align*}
The set of prepictures between $P$ and $Q$ is denoted by $Pic_<(P,Q)$.
\item A picture (or standard picture) between $P$ and $Q$ is a bijection $f:V(P)\longrightarrow V(Q)$ such that:
\begin{align*}
\forall i,j \in V(P),\:&i\leq_1 j\Longrightarrow f(i)\leq_2 f(j),&&i<_1 j\Longrightarrow f(i)<_2 f(j),\\
&f(i)\leq_1f(j)\Longrightarrow i\leq_2 j,&&f(i)<_1f(j)\Longrightarrow i<_2 j.
\end{align*}
The set of pictures between $P$ and $Q$ is denoted by $Pic(P,Q)$.
\item A semi-standard picture between $P$ and $Q$ is a bijection $f:V(P)\longrightarrow V(Q)$ such that:
\begin{align*}
\forall i,j \in V(P),\:&i<_1 j\Longrightarrow f(i)\leq_2 f(j),f(i)<_1f(j)\Longrightarrow i\leq_2 j.
\end{align*}
The set of semi-standard pictures between $P$ and $Q$ is denoted by $Pic_{ss}(P,Q)$.
\end{itemize}\end{defi}

\textbf{Remarks.} \begin{enumerate}
\item  Obviously, $Pic(P,Q)\subseteq Pic_<(P,Q)$; moreover:
$$Pic_<(P,Q)=Pic_<(pos(P),pos(Q))=Pic(pos(P),pos(Q)).$$
\item If $\leq_2$ are orders for both $P$ and $Q$, then any bijection $f:V(P)\longrightarrow V(Q)$ is a picture between $P$ and $Q$ if, and only if:
\begin{align*}
\forall i,j \in V(P),\:&i\leq_1 j\Longrightarrow f(i)\leq_2 f(j),&&f(i)\leq_1f(j)\Longrightarrow i\leq_2 j.
\end{align*}\end{enumerate}

\textbf{Example.} Let $P,Q\in\tqp(n)$. Then,
$$Pic(P,Q)=Pic_<(P,Q)=Pic_{ss}(P,Q)\cong S_n.$$
This generalizes the correspondence between permutations and pictures of trivial double posets,
instrumental in the picture-theoretical reformulation of the RS correspondence between permutations and pairs of standard tableaux.

\textbf{Example.} Let $\lambda$ be a Young diagram with $n$ entries, embedded in $\N^2$. We write $Q_\lambda$ (resp. $P_\lambda$, for later use) for the double quasi-poset with $V(Q_\lambda)=\lambda$, equipped with the order $(x,y)\leq (z,t)\Leftrightarrow x\leq z \ {\text {and}} \ y\leq t$ and an arbitrary preorder $\leq_2$, respectively the strict order $\leq_2$ obtained by labelling the entries of $\lambda$ in the reading order: graphically, for $$\lambda=\young(\ ,\ \ \ ,\ \ \ ), \ \leq_2 \text{is given by }
\young(1,234,567).$$
\begin{enumerate}
\item Let $Q\in \tqp$ with $V(Q)=[n]$ and $\leq_2$ the natural order. Then, $Pic(Q_\lambda,Q)=Pic_<(Q_\lambda,Q)=Pic_{ss}(Q_\lambda,Q)$ is in bijection with standard tableaux of shape $\lambda$.
\item Let $Q\in \tqp$ with $V(Q)=[n]$ and $\leq_2$ an arbitrary  order. Then, $Pic(Q_\lambda,Q)=Pic_<(Q_\lambda,Q)=Pic_{ss}(Q_\lambda,Q)$ is in bijection with tableaux of shape $\lambda$ such that the entries are increasing from bottom to top and left to right for $\leq_2$.
\item Let $Q\in \tqp$ with $V(Q)=[n]$ and $\leq_2$ be an arbitrary preorder. Then, $Pic(Q_\lambda,Q)=Pic_<(Q_\lambda,Q)$ is in bijection with  tableaux of shape $\lambda$ such that the entries are strictly increasing from bottom to top and left to right for $\leq_2$.
Instead, $Pic_{ss}(Q_\lambda,Q)$ is in bijection with tableaux of shape $\lambda$ such that the entries are weakly increasing from bottom to top and left to right for $\leq_2$.
\end{enumerate}

\begin{lemma}
Let $P,Q\in\dqp$. The sets $Pic(P,Q),\ Pic_<(P,Q),\ Pic_{ss}(P,Q)$ are $Aut(P)^{op}$ (resp. $Aut(Q)$)-sets
by right (resp. left) composition, where $Aut(P)^{op}$ is the opposite of the group of automorphisms of $P$.
\end{lemma}
\begin{defi}
Let $P,Q\in \dqp$, two bijections $f,g$ from $V(P)$ to $V(Q)$ are called equivalent (written $f\sim g$) if, and only if, there exists
$(\phi,\psi)\in Aut(P)\times Aut(Q)$ such that $f=\psi \circ g\circ \phi$. The quotient $Pat(P,Q):=Aut(Q)\setminus Pic_{ss}(P,Q)/Aut(P)$ is called the set of patterns between $P$ and $Q$.
\end{defi}

\textbf{Example.} The notations are as in the previous example, we assume furthermore that $Q_\lambda=P_\lambda$ and that, on $Q$, the preorder
 $\leq_2$ is total and increasing (e.g. $1\sim_2 2<_2 3\sim_2 4\sim_2 5<_2 6\sim_2 7<_2 8$) and identifies therefore with a surjection ($f(1)=f(2)=1,f(3)=f(4)=f(5)=2,f(6)=f(7)=3, f(8)=4)$), resp. an increasing packed word ($11222334$), resp. a composition ($\mathbf{n}:=(2,3,2,1)$). For later use, we also set $Q(\mathbf{n}):=Q$. Then, $Aut(P_\lambda)=\{1\}$ (because the second order on $P_\lambda$ is strict) and $Aut(Q(\mathbf{n}))$ is, up to a canonical isomorphism, a Young subgroup of $S_n$ (e.g. $Aut(Q(\mathbf{n}))\cong S_\mathbf{n} :=S_2\times S_3\times S_2\times S_1$). Therefore, $Pat(P_\lambda,Q(\mathbf{n}))$ is in bijection with the set $\lambda(\mathbf{n})$ of tableaux of shape $\lambda$ decorated by the packed word $11222334$ and such that the entries are row and column-wise weakly increasing: for example,  
$$\mbox{for }\lambda=\young(\ \ ,\ \ \ ,\ \ \ ) \ \  \text{, a tableau such as }
\young(24,223,113)\:.$$

 \begin{lemma}
For any composition $\mathbf{n}$ of $n$, $Pat(Q(\mathbf n),P_{[n]})$ is in bijection with the set of surjections $f$ from $[n]$ to $[k]$
 such that $|f^{-1}(1)|=n_1,\dots,|f^{-1}(k)|=n_k$.
\end{lemma}

When $\mathbf{n}=(1,\dots,1)$, $Q(\mathbf n)=P_{[n]}$, sets of patterns and pictures identify,  and we recover 
$Pat(Q(\mathbf n),P_{[n]})=Pic(Q(\mathbf n),P_{[n]})\cong S_n$.

In general, $Pic_{ss}(Q(\mathbf n),P_{[n]})\cong S_n$ and $Aut(Q(\mathbf n))\cong S_{\mathbf n}:=S_{n_1}\times \dots\times S_{n_k}$. The result follows from the usual bijection between the coset $S_n/ S_{\mathbf n}$ and the set of surjections from $[n]$ to $[k]$ such that $|f^{-1}(1)|=n_1,\dots,|f^{-1}(k)|=n_k$.

\section{Pairings and self-duality}\label{psdu}
We depart from now on from diagrammatics and topics such as the RS correspondence to focus on the algebraic structures underlying the theory of pictures for double quasi-posets. The present section investigates duality phenomena.

\begin{lemma}
For all $P,Q,R\in \dqp$ and $f\in Pic_<(PQ,R)$ (resp. $f\in Pic(PQ,R)$), we have $f(V(Q))\in Top_<(R)$ (resp. $Top(R)$).
\end{lemma}
\begin{proof}
We put $O=f(V(Q))$. Let $i',j'\in V(R)$, with $i'\in O$ and $i'<_1 j'$.
We put $i'=f(i)$ and $j'=f(j)$. Then $i\in V(Q)$ and $f(i)<_1 f(j)$, so $i<_2 j$, with $j\in V(Q)$, and finally $j'\in O$: $O\in Top_<(R)$. The same argument applies \it mutatis mutandis \rm for pictures. \end{proof}

\begin{prop}
For all $P,Q\in \dqp$, we put:
\begin{align*}
\langle P,Q\rangle&=\sharp Pic(P,Q),&\langle P,Q\rangle_<&=\sharp Pic_<(P,Q).
\end{align*}
$\langle-,-\rangle_<$  and $\langle-,-\rangle$ are symmetric Hopf pairings on, respectively, $(\h_\dqp,m,\Delta_<)$ and $(\h_\dqp,m,\Delta)$.
Moreover, for all $x,y\in \h_\dqp$:
\begin{align*}
\langle x,y\rangle_<&=\langle pos(x),pos(y)\rangle_<=\langle pos(x),pos(y)\rangle.
\end{align*}
\end{prop}

\begin{proof} Let $P,Q \in \dqp$. The map $f\mapsto f^{-1}$ is a bijection from $Pic_<(P,Q)$ to $Pic_<(Q,P)$ and from $Pic(P,Q)$ to $Pic(Q,P)$.
So $\langle P,Q\rangle_<=\langle Q,P\rangle_<$ and $\langle P,Q\rangle=\langle Q,P\rangle$.\\

Let $P,Q,R\in \dqp$, we set ${O}^c:=V(R)\setminus O$ and define:
\begin{align*}
\theta:&\left\{\begin{array}{rcl}
Pic_<(PQ,R)&\longrightarrow&\displaystyle \bigsqcup_{O\in Top_<(R)} Pic_<(P,R_{\mid O^c})\times Pic_<(Q,R_{\mid O})\\
f&\longrightarrow&(f_{\mid P},f_{\mid Q}).
\end{array}\right.
\end{align*}
By the previous Lemma, and since by restriction, for $O:=f(V(Q))$, $f_{\mid P} \in Pic_<(P,R_{\mid O^c})$ and $f_{\mid Q}\in Pic_<(Q,R_{\mid O})$, $\theta$ is well-defined. By its very definition, it is injective. 
Let $O\in Top_<(R)$, $(f_1,f_2)\in  Pic_<(P,R_{\mid O^c})\times Pic_<(Q,R_{\mid O})$.
We denote by $f$ the unique bijection from $V(PQ)$ to $V(R)$ such that $f_{\mid P}=f_1$ and $f_{\mid Q}=f_2$. We let the reader check that $f\in Pic_<(PQ,R)$; $\theta$ is bijective and we obtain:
\begin{align*}
\langle PQ,R\rangle_<&=\sharp Pic_<(PQ,R)\\
&=\sum_{O\in Top_<(R)} \sharp Pic_<(P,R_{\mid O^c})\sharp Pic_<(Q,R_{\mid O})\\
&=\sum_{O\in Top_<(R)} \langle P,R_{\mid O^c})\rangle_<\langle Q,R_{\mid O}\rangle_<\\
&=\langle P\otimes Q,\Delta_<(R)\rangle_<.
\end{align*}
So $\langle-,-\rangle_<$ is a Hopf pairing. By restriction, one gets:
$$\theta(Pic(PQ,R))=\bigsqcup_{O\in Top(R)} Pic(P,R_{\mid O^c})\times Pic(Q,R_{\mid O});$$
that $\langle-,-\rangle$ is a Hopf pairing follows by similar arguments that we omit.

Moreover, for any $P,Q\in \dqp$:
\begin{align*}
\langle P,Q\rangle_<&=\sharp Pic_<(P,Q)=\sharp Pic_<(pos(P),pos(Q))=\sharp Pic(pos(P),pos(Q))\\
&=\langle pos(P),pos(Q)\rangle.
\end{align*}\end{proof}

\begin{defi}
The map $\iota:\h_\dqp\longrightarrow \h_\dqp$ is defined by $\iota(P)=(V(P),\leq_2,\leq_1)$ for any $P=(V(P),\leq_1,\leq_2)\in \dqp$.
\end{defi}

\begin{lemma}
For any double quasi-poset $P$, we put:
\begin{align*}
X_P&=\{(i,j)\in V(P)\mid i\leq_1 j\},&x_P&=\sharp X_P,\\
Y_P&=\{(i,j)\in V(P)\mid i\leq_2 j\},&y_P&=\sharp Y_P.
\end{align*}
\begin{enumerate}
\item Let $P,Q\in \dqp$, such that $\langle P,Q\rangle\neq 0$. Then:
\begin{itemize}
\item $x_P\leq y_Q$ and $x_Q\leq y_P$.
\item If moreover $x_P=y_Q$ and $x_Q=y_P$, then $Q=\iota(P)$.
\end{itemize}
\item For any $P\in \dqp$, $\langle P,\iota(P)\rangle\neq 0$.
\end{enumerate}\end{lemma}

\begin{proof}
1. The set $Pic(P,Q)$ is non-empty. Let $f\in Pic(P,Q)$. We define:
\begin{align*}
F:&\left\{\begin{array}{rcl}
V(P)^2&\longrightarrow&V(Q)^2\\
(i,j)&\longrightarrow&(f(i),f(j)).
\end{array} \right.
\end{align*} 
As $f$ is bijective, $F$ is bijective. By definition of a picture, $F(X_P)\subseteq Y_Q$ and $F^{-1}(X_Q)\subseteq Y_P$,
so $x_P\leq y_Q$ and $x_Q\leq y_P$. If moreover $x_P=y_Q$ and $x_Q=y_P$, then $F(X_P)=Y_Q$ and $F^{-1}(X_Q)=Y_P$;
for any $i,j \in V(P)$:
\begin{align*}
i\leq_1 j \mbox{ in }P&\Longleftrightarrow f(i) \leq_2 f(j)\mbox{ in }Q\Longleftrightarrow f(i) \leq_1 f(j)\mbox{ in }\iota(Q)\\
f(i) \leq_2 f(j)\mbox{ in }\iota(Q)&\Longleftrightarrow f(i) \leq_1 f(j) \in Q\Longleftrightarrow i\leq_2 j\mbox{ in }P.
\end{align*} 
So $f$ is an isomorphism from $P$ to $\iota(Q)$: $P=\iota(Q)$ or, equivalently $Q=\iota(P)$.\\

2. If $P=\iota(Q)$, then $Id_{V(P)} \in Pic(P,Q)$, so $\langle P,Q\rangle\neq 0$. \end{proof}

\begin{prop}\label{propnd}
Let $\mathbf{X}\subseteq \dqp$ such that:
\begin{itemize}
\item $1\in \mathbf{X}$.
\item $x,y\in \mathbf{X}\Longrightarrow xy\in \mathbf{X}$.
\item $\forall P\in \mathbf{X}$, $\forall B\subseteq V(P)$, $P_{\mid B} \in \mathbf{X}$.
\end{itemize}
We denote by $\h_\mathbf{X}$ the subspace of $\h_\dqp$ generated by $\mathbf{X}$. Then $\h_\mathbf{X}$ 
is a Hopf subalgebra of both $(\h_\dqp,m,\Delta)$ and $(\h_\dqp,m,\Delta_<)$. If, moreover:
\begin{itemize}
\item $x\in \mathbf{X} \Longrightarrow \iota(x)\in \mathbf{X}$,
\end{itemize}
then  $\langle-,-\rangle_{\mid \h_\mathbf{X}}$ is non-degenerate, so $(\h_\mathbf{X},m,\Delta)$ is a graded self-dual Hopf algebra.
\end{prop}

\begin{proof} Let us fix an integer $n$. We denote by $\dqp(n)$, respectively $\mathbf{X}(n)$, the set of double quasi-posets of order $n$,
respectively $\mathbf{X}\cap \dqp(n)$.
We define a relation $\leq$ on $\dqp(n)$ by:
\begin{align*}
\forall P,Q\in \dqp(n),\:P\prec Q\mbox{ if }&(P=Q) \\
&\mbox{ or }( (x_P,y_P)\neq (x_Q,y_Q),\: x_P\leq x_Q\mbox{ and }y_P\geq y_Q).
\end{align*}
This is a preorder on $\dqp(n)$. Let us assume that $P\prec Q$ and $Q\prec P$. Then $x_P\leq x_Q\leq x_P$ and $y_P\geq y_Q\geq y_P$, 
so $(x_P,y_P)=(x_Q,y_Q)$, which implies that $P=Q$: we proved that $\prec$ is an order on $\dqp(n)$. 
We now consider a linear extension $\leq$ of $\prec$. In other words, $\leq$ is a total order on $\dqp(n)$
such that for any $P,Q\in \dqp(n)$:
$$((x_P,y_Q)\neq (x_Q,y_Q),\: x_P\leq x_Q\mbox{ and }y_P\geq y_Q)\Longrightarrow P\leq Q.$$
Let us assume that $\langle P,Q\rangle\neq 0$. By the preceding lemma, $x_P\leq y_Q$ and $y_P\geq x_Q$, 
so $x_P\leq x_{\iota(Q)}$ and $y_P\geq y_{\iota(Q)}$. If $(x_P,y_P)\neq (y_Q,x_Q)$, then $P\geq \iota(Q)$;
if $(x_P,y_P)=(y_Q,x_Q)$, by the preceding lemma $P=\iota(Q)$. Finally:
$$\langle P,Q\rangle\neq 0\Longrightarrow P\geq \iota(Q).$$

We now write $\mathbf{X}(n)=\{P_1,\ldots,P_k\}$ in such a way that $\iota(P_1)\leq \ldots \leq \iota(P_k)$. The matrix 
$(\langle \iota(P_i),P_j\rangle)_{1\leq i,j\leq k}$ is upper triangular, and its diagonal terms are the elements $\langle P_i,\iota(P_i)\rangle$, 
which are non-zero by the preceding lemma: this matrix is invertible. Hence, $\langle-,-\rangle_{\mid \h_\mathbf{X}}$ is non-degenerate.\end{proof}

This can be applied with $\mathbf{X}=\dqp$ or $\mathbf{X}=\dpos$.

\begin{cor}
The pairings $\langle-,-\rangle$ and $\langle-,-\rangle_{\mid \h_\dpos}$ are non-degenerate. The kernel of $\langle-,-\rangle_<$ is $Ker(pos)$.
\end{cor}

\begin{proof}
The families $\dqp$ and $\dpos$ satisfy the hypotheses of proposition \ref{propnd}, so $\langle-,-\rangle$ and $\langle-,-\rangle_{\mid \h_\dpos}$
are non-degenerate. 
As $Im(pos)=\h_\dpos$ and $\langle-,-\rangle_{\mid \h_\dpos}$ is non-degenerate, $Ker(pos)=Ker(\langle-,-\rangle_<)$. \end{proof}

\begin{prop} \label{propcouplage}
For all $x,y\in \h_\sqp$:
\begin{align*}
\langle \Upsilon(x),\Upsilon(y)\rangle&=\langle x,y\rangle_<.
\end{align*}
\end{prop}

\begin{proof}
The identity $\langle \Upsilon(x),\Upsilon(y)\rangle=\langle x,y\rangle_<$ is better understood as the consequence of a bijection 
from pictures between the blow ups of special double quasi-posets $P$ and $Q$, to the prepictures between $P$ and $Q$.  We sketch the proof.
Let $P'$ and $Q'$ be blow ups of $P$ and $Q$. Since $i<_1 j$ in $P$ implies $i<_1 j$ in $P'$, and similarly for $Q$ and $Q'$, a picture between $P'$ and $Q'$ is a prepicture between $P$ and $Q$.
Let conversely $f$ be a prepicture between $P$ and $Q$. Let $\{P_i\}_{i=1\dots k}$ be the set of all nontrivial equivalence class for $\sim_1$ in $P$. The inverse image by $f$ of the preorder $\leq_2$ on each $f(P_i)\subset V(Q)$ is a preorder $\leq^i$ on $P_i$. Blowing up $P$ successively along $\leq^1,\dots,\leq^k$ defines a blow up $P'$ of $P$. By symmetry, a blow up $Q'$ of $Q$ is defined by the same process. By construction, $f$ is a picture between $P'$ and $Q'$. \end{proof}

\section{Internal products}\label{ip}
This section addresses the question of internal products. The existence of internal products (by which we mean the existence of an associative product on double posets with a given cardinality) is a classical property of combinatorial Hopf algebras: in the representation theory of the symmetric group (or equivalently in the algebra of symmetric functions) the internal product is obtained from the tensor product of representations, and this product extends naturally to various noncommutative versions, such as the descent algebra or the Malvenuto-Reutenauer Hopf algebra \cite{MR}. 

The rich structure of double quasi-posets allows for the definition of two internal associative products generalizing the corresponding structures on double posets \cite{MR2}.

\begin{defi}Let $P,Q\in \dqp$ and $f:V(P)\longrightarrow V(Q)$ a bijection.
We define a double quasi-poset, the product of $P$ and $Q$ over $f$, $P\times_f Q=(V(P),\leq_1^f,\leq_2)$ by:
\begin{align*}
\forall i,j\in V(P),\:&i\leq^f_1 j\mbox{ if }f(i)\leq_1 f(j),
\end{align*}
where $\leq_2$ is the second preorder on $P$.
\end{defi}
\begin{defi}
Let $P,Q\in \dqp$ and $f:V(P)\longrightarrow V(Q)$ a bijection.
\begin{enumerate}
\item We shall say that $f$ is a semi-prepicture between $P$ and $Q$ if:
\begin{align*}
\forall i,j\in V(P),\:&i<_1 j\Longrightarrow f(i)<_2 f(j).
\end{align*}
The set of semi-prepictures between $P$ and $Q$ is denoted by $I_<(P,Q)$.
\item We shall say that $f$ a semi-picture between $P$ and $Q$ if:
\begin{align*}
\forall i,j\in V(P),\:&i<_1 j\Longrightarrow f(i)<_2 f(j),&i\leq_1 j\Longrightarrow f(i)\leq_2 f(j).
\end{align*}
The set of semi-pictures between $P$ and $Q$ is denoted by $I(P,Q)$.
\end{enumerate} \end{defi}

\textbf{Remark.} For any $P,Q\in \dqp$:
\begin{align*}
Pic_<(P,Q)&=I_<(P,Q)\cap I_<(Q,P)^{-1},&Pic(P,Q)&=I(P,Q)\cap I(Q,P)^{-1}.
\end{align*}

\begin{prop}
For any double quasi-posets $P$, $Q$, we put:
\begin{align*}
P\unlhd Q&=\sum_{f\in I(P,Q)} P\times_f Q,&P\lhd Q&=\sum_{f\in I_<(P,Q)} P\times_f Q.
\end{align*}
These products are bilinearly extended to $\h_\dqp$. Then both $\unlhd$ and $\lhd$ are associative and, for all $x,y,z\in \h_\dqp$:
\begin{align*}
\langle x\unlhd y,z\rangle&=\langle x,y\unlhd z\rangle,&\langle x\lhd y,z\rangle_<&=\langle x,y\lhd z\rangle_<.
\end{align*}\end{prop}

\begin{proof} \textit{First step.} Let us first prove the associativity of $\lhd$. Let $P,Q,R \in \dqp$. We consider:
\begin{align*}
X&=\{(f,g)\mid f\in I_<(P,Q),\:g\in I_<(P\times_f Q,R)\},\\
X'&=\{(f',g')\mid f'\in I_<(Q,R),\:g'\in I_<(P,Q\times_{f'} R)\}.
\end{align*}
We consider the maps:
\begin{align*}
\phi:&\left\{\begin{array}{rcl}
X&\longrightarrow&X'\\
(f,g)&\longrightarrow&(g\circ f^{-1},f),
\end{array}\right.&
\phi'&:\left\{\begin{array}{rcl}
X'&\longrightarrow&X\\
(f',g')&\longrightarrow&(g',f'\circ g').
\end{array}\right.
\end{align*}
Let us prove that they are well-defined. Let us take $(f,g)\in X$; we put $(f',g')=(g\circ f^{-1},f)$. If $i<_1j$ in $Q$,
then $f^{-1}(i)<_1^f f^{-1}(j)$ in $P\times_f Q$, so $f'(i)=g\circ f^{-1}(i)<_2 g\circ f^{-1}(j)=f'(j)$. 
If $i<_1 j$ in $P$, then $g'(i)=f(i)<_2 f(j)=g'(j)$ in $Q$, or equivalently in $Q\times_{f'} R$. So $\phi$ is well-defined.

Let us take $(f',g')\in X'$; we put $(f,g)=(g',f'\circ g')$. If $i<_1j$ in $P$, then $f(i)=g'(i)<_2g'(j)=f(j)$ in $Q\times_{f'} R$, so in $Q$.
If $i<^f_1 j$ in $P\times_f Q$, then $g'(i)=f(i)<_1 f(j)=g'(j)$ in $Q$, so $g(i)=f'\circ g'(i)<_2 f'\circ g'(j)$ in $R$. So $\phi'$ is well-defined. 

It is immediate to prove that $\phi\circ \phi'=Id_{X'}$ and $\phi'\circ \phi=Id_X$. We get finally:
\begin{align*}
(P\lhd Q)\lhd R&=\sum_{(f,g)\in X} (P\times_f Q)\times_g R\\
&=\sum_{(f',g')\in X'} P\times_{f'}(Q\times_{f'} R)=P\lhd(Q\lhd R).
\end{align*}
The associativity of $\unlhd$ is proved in the same way: we consider
\begin{align*}
X''&=\{(f,g)\mid f\in I(P,Q),\:g\in I(P\times_f Q,R)\}\subseteq X,\\
X'''&=\{(f',g')\mid f'\in I(Q,R),\:g'\in I(P,Q\times_{f'} R\}\subseteq X'.
\end{align*}
The same computations as before, replacing everywhere $<$ by $\leq$, proving that $\phi(X'')=X'''$, allow to conclude similarly.\\

\textit{Second step}. Let $P,Q,R \in \dqp$. We consider:
\begin{align*}
Y&=\{(f,g)\mid f\in I_<(P,Q),\:g\in Pic_<(P\times_f Q,R)\}\subseteq X,\\
Y'&=\{(f',g')\mid f'\in I_<(Q,R),\:g'\in Pic_<(P,Q\times_{f'} R)\}\subseteq X'.
\end{align*}
Let us prove that $\phi(Y)=Y'$. Let us take $(f,g)\in Y$; we put $(f',g')=\phi(f,g)$. If $g'(i)<_1 g'(j)$ in $Q\times_{f'} R$,
then $f(i)<_1 f(j)$ in $Q\times_{g\circ f^{-1}}R$, so $g(i)<_1 g(j)$ in $R$. As $g\in Pic_<(P\times_f Q,R)$,
$i<_2 j$ in $P\times_f Q$ or equivalently in $P$. If $(f',g')\in Y'$, we put $\phi'(f',g')=(f,g)$.
If $g(i)<_1 g(j)$ in $R$, then $f'\circ g'(i)<_1 f'\circ g'(j)$ in $R$, so $g'(i)<_1^{f'} g'(j)$ in $Q\times_{f'} R$. As $g'\in Pic_<(P,Q\times_{f'} R)$,
$i<_2 j$ in $P$, or in $P\times_f Q$.

Consequently:
\begin{align*}
\langle P\lhd Q,R\rangle_<&=\sharp Y=\sharp Y'=\langle P,Q\lhd R\rangle_<.
\end{align*}

Putting:
\begin{align*}
Y''&=\{(f,g)\mid f\in I(P,Q),\:g\in Pic(P\times_f Q,R)\}\subseteq X'',\\
Y'''&=\{(f',g')\mid f'\in I(Q,R),\:g'\in Pic(P,Q\times_{f'} R)\}\subseteq X''',
\end{align*}
we prove in the same way, replacing everywhere $<$ by $\leq$, that $\phi(Y'')=Y'''$. So $\langle P\unlhd Q,R\rangle=\langle P,Q\unlhd R\rangle$.
\end{proof}

\textbf{Remark.} $\h_\dpos$ and $\h_\sqp$ are stable under $\unlhd$ and $\lhd$.

\begin{prop}
For all $x,y\in \h_\sqp$:
\begin{align*}
\Upsilon(x\lhd y)&=\Upsilon(x)\unlhd \Upsilon(y).
\end{align*}
\end{prop}

\begin{proof}
The same argument as in the proof of proposition \ref{propcouplage}, used with semi-pictures and semi-prepictures shows that 
$P\lhd Q=\Upsilon(P)\unlhd Q$. The identity $\Upsilon(P\lhd Q)=\Upsilon(P)\unlhd \Upsilon(Q)$ follows by noticing that $\Upsilon$ maps a double quasi-posets to the sum of its blow ups and that, since for $P'\unlhd Q$ with $P'\in B(P)$, $\leq_1$ is obtained as the inverse image along a semi-picture  of the preorder $\leq_1$ on $Q$, $\Upsilon(\Upsilon(P)\unlhd Q)=\Upsilon(P)\unlhd \Upsilon(Q)$.
\end{proof}

\section{Permutations and surjections}\label{ps}
In this section, we study the restriction of the internal products on double quasi-posets to the linear spans of surjections $k\E_n$. 
One product ($\unlhd$) identifies essentially with the naive composition product of surjections, but the other one, $\lhd$, that emerges naturally from the theory of pictures,
is  not induced by the composition of surjections and differs from the product in the Solomon-Tits algebra; recall that the latter is an algebra structure on ordered set partitions of $[n]$  -- that can be identified bijectively with surjections -- it emerges naturally from the theory of twisted Hopf algebras, also called Hopf species \cite{patras2006twisted}.

Let $w=w(1)\ldots w(n)=:w_1\dots w_n$ be a surjection from $[n]$ to $[k]$, or equivalently a packed word of length $n$ with $k$ distinct letters. We define a special double poset $P_w$ by:
\begin{enumerate}
\item $V(P_w)=\{1,\ldots,n\}$.
\item $\forall i,j\in \{1,\ldots,n\}$, $i\leq_1 j$ if $w_i\leq w_j$.
\item $\leq_2$ is the usual order on $\{1,\ldots,n\}$.
\end{enumerate}
We obtain in this way an injection from the set of surjections $\E_n$ to $\sqp$.

\begin{defi}
Let $w$ be a packed word of length $n$ and $\sigma \in \mathfrak{S_n}$. We shall say that $\sigma$ is $w$-compatible if:
\begin{align*}
\forall i,j\in \{1,\ldots,n\},\:& w_i< w_j \Longrightarrow \sigma(i)<\sigma(j).
\end{align*}
The set of $w$-compatible permutations is denoted by $Comp(w)$.
\end{defi}

\textbf{Remark}. If $w$ is a permutation, $Comp(w)=\{w\}$. In general:
\begin{align*}
\sharp Comp(w)&=\prod_{i=1}^{\max(w)} (\sharp w^{-1}(i))!.
\end{align*}

\begin{prop}
Let $u,v$ be two packed words of the same length $n$. Then:
\begin{align*}
P_u\lhd P_v&=\sum_{\sigma \in Comp(u)} P_{v\circ \sigma},&
P_u\unlhd P_v&=\begin{cases}
v\circ u\mbox{ if }u\in \mathfrak{S}_n,\\
0\mbox{ otherwise}.
\end{cases}\end{align*}
\end{prop}

\begin{proof} Let $\sigma \in \mathfrak{S}_n$. Then $\sigma \in I_<(P_u,P_v)$ if, and only if, $\sigma \in Comp(u)$.
Moreover, $\sigma \in I(P_u,P_v)$ if, and only if, $\sigma=u$. Hence:
\begin{align*}
P_u\lhd P_v&=\sum_{\sigma \in Comp(u)} P_u\times_\sigma P_v,&
P_u\unlhd P_v&=\begin{cases}
P_u \times_u P_v\mbox{ if }u\in \mathfrak{S}_n,\\
0\mbox{ otherwise}.
\end{cases}\end{align*}
Moreover, for all $1\leq i,j\leq n$:
\begin{align*}
i\leq_1^\sigma j \mbox{ in }P_u\times_\sigma P_v&\Longleftrightarrow \sigma(i) \leq_1 \sigma(j)\mbox{ in }P_v
\Longleftrightarrow v\circ\sigma(i)\leq v\circ \sigma(j).
\end{align*}
So $P_u\times_\sigma P_v=P_{v\circ \sigma}$.\end{proof}

\textbf{Remark.} In particular, if $u,v\in \mathfrak{S}_n$, $P_u\lhd P_v=P_u\unlhd P_v=P_{v\circ u}$.

\begin{prop}
The following maps are algebra morphisms:
\begin{align*}
\zeta:&\left\{\begin{array}{rcl}
(k\E_n,\lhd)&\longrightarrow&(k\mathfrak{S}_n,\circ)\\
P_w&\longrightarrow&\displaystyle \sum_{\sigma\in Comp(w)} \sigma^{-1},
\end{array}\right.&
\zeta':&\left\{\begin{array}{rcl}
(k\mathfrak{S}_n,\circ)&\longrightarrow&(k\E_n,\lhd)\\
\sigma&\longrightarrow&P_{\sigma^{-1}}.
\end{array}\right.
\end{align*}
Moreover, the following diagram commutes:
$$\xymatrix{k\mathfrak{S}_n\ar[r]^{\zeta'} \ar[rd]_{Id}&k\E_n\ar[d]^{\zeta}\\
&k\mathfrak{S}_n}$$
\end{prop}

\begin{proof} Let $u,v$ be packed words of length $n$. For any $\sigma,\tau \in \mathfrak{S}_n$:
\begin{align*}
\tau \in Comp(v\circ \sigma)&\Longleftrightarrow (\forall \:1\leq i,j\leq n,\: v\circ \sigma(i)<v\circ \sigma(j)\Longrightarrow \tau(i)<\tau(j))\\
&\Longleftrightarrow (\forall \: 1\leq i,j\leq n,\: v(i)<v(j)\Longrightarrow \tau\circ \sigma^{-1}(i)<\tau\circ \sigma^{-1}(j))\\
&\Longleftrightarrow \tau \circ \sigma^{-1} \in Comp(v).
\end{align*}
Hence:
\begin{align*}
\zeta(P_u \lhd P_v)&=\sum_{\sigma \in Comp(u), \tau \in Comp(v\circ \sigma)} \tau^{-1}\\
&=\sum_{\sigma \in Comp(u), \tau\in Comp(v)}(\tau \circ \sigma)^{-1}\\
&=\sum_{\sigma \in Comp(u), \tau\in Comp(v)}\sigma^{-1} \circ \tau^{-1}\\
&=\zeta(P_u)\circ \zeta(P_v).
\end{align*}
So $\zeta$ is an algebra morphism. \end{proof}

We conclude by comparing the two products on $k\E_n$.

\begin{prop}The map $\Upsilon$ restricts to an algebra isomorphism
$$\Upsilon:(k\E_n,\lhd)\longrightarrow (k\E_n,\unlhd).$$
\end{prop}
\begin{proof}
Recall that, on $\h_\sqp$, 
$$\Upsilon(x\lhd y)=\Upsilon(x) \unlhd \Upsilon(y).$$
Let $P=(V(P),\leq_1,\leq_2)\in \dqp$. There exists a surjection or packed word $u$ such that $P=P_u$ if, and only if,
$\leq_2$ is a total order and $\leq_1$ is a total preorder.
The class of these double posets originating from surjections is stable by blow ups, the Proposition follows. \end{proof}


\bibliographystyle{amsplain}
\bibliography{biblio}

\end{document}